\documentclass[12pt]{amsart}
\usepackage{enumerate,amsmath,amssymb,bm}
\usepackage{url}
\usepackage{xspace}

\usepackage[margin=1in]{geometry}
\DeclareMathOperator{\pnt}{\raise 0.5mm \hbox{\large\textbf{.}}}

\newcommand{\note}[2][ ]{}%dummy macro

\newtheorem{theorem}{Theorem}%[section]
\newtheorem{lemma}[theorem]{Lemma}
\newtheorem{proposition}[theorem]{Proposition}

\theoremstyle{definition}

\newtheorem{question}[theorem]{Question}

\usepackage[pdfauthor={William J. keith and Fabrizio Zanello},
            pdftitle={Log-concavity of level Hilbert functions and pure $O$-sequences},
            pdfsubject={enumerative combinatorics},
            pdfkeywords={regular partitions, parity},
            pdfproducer={Latex with hyperref},
            pdfcreator={latex->dvips->ps2pdf},
            pdfpagemode=UseNone,
            bookmarksopen=false,
            bookmarksnumbered=true]{hyperref}

\begin{document}
\title[Log-concavity of level Hilbert functions and pure $O$-sequences]{Log-concavity of level Hilbert functions\\and pure $O$-sequences}
\author{Fabrizio Zanello} \address{Department of Mathematical Sciences\\ Michigan Tech\\ Houghton, MI  49931}
\email{zanello@mtu.edu}

\maketitle
%\linenumbers

\begin{abstract}
We investigate log-concavity in the context of level Hilbert functions and pure $O$-sequences, two classes of numerical sequences introduced by Stanley in the late Seventies whose structural properties have since been the object of a remarkable amount of interest in combinatorial commutative algebra. However, a systematic study of the log-concavity of these sequences began only recently, thanks to a paper by Iarrobino.

The goal of this note is to address two general questions left open by Iarrobino's work: 1) Given the integer pair $(r,t)$, are all level Hilbert functions of codimension $r$ and type $t$ log-concave? 2) How about pure $O$-sequences with the same parameters? 

Iarrobino's main results consisted of a positive answer to 1) for $r=2$ and any $t$, and for $(r,t)=(3,1)$. Further, he proved that the answer to 1) is negative for $(r,t)=(4,1)$.

Our chief contribution to 1) is to provide a negative answer in all remaining cases, with the exception of $(r,t)=(3,2)$, which is still open in any characteristic. We then propose a few detailed conjectures specifically on level Hilbert functions of codimension 3 and type 2. 

As for question 2), we show that the answer is positive for all pairs $(r,1)$; negative for $(r,t)=(3,4)$; and negative for any pair $(r,t)$ with $r\ge 4$ and $2\le t\le r+1$. Interestingly, the main case that remains open is again $(r,t)=(3,2)$. Further, we conjecture that, in analogy with the behavior of arbitrary level Hilbert functions, log-concavity fails for pure $O$-sequences of any codimension $r\ge 3$ and type $t$ large enough.
\end{abstract}

\section{Introduction and main definitions}

This note is inspired by the recent paper \cite{Ia2}, where Iarrobino began investigating the log-concavity of the Hilbert functions of Gorenstein and level artinian graded algebras, in response to a question posed by Chris McDaniel. Surprisingly, despite the substantial amount of literature devoted to understanding the algebraic, combinatorial, and geometric properties of these Hilbert functions, before \cite{Ia2} little attention was paid to log-concavity, an important numerical property with applications to several mathematical fields (see, e.g., the classical surveys \cite{Bre,Stan}). Recall that a sequence $(a_n)_{n\in \mathbb N}$ is said to be \emph{log-concave} if $a_{i-1}a_{i+1}\leq a_i^2$, for all indices $i$.

Iarrobino's main results are that log-concavity holds for all Gorenstein Hilbert functions of codimension 3, while it fails in codimension 4 (see below, as well as \cite{Ia2}, for the relevant terminology). Also, log-concavity holds for all level Hilbert functions of codimension 2. Thus, two general and natural questions that arise from \cite{Ia2} are: 1) Given the integer pair $(r,t)$, are all level Hilbert functions of codimension $r$ and type $t$ log-concave? 2) How about pure $O$-sequences (i.e., \emph{monomial} level Hilbert functions) with the same parameters?

The goal of this note is to make progress on both questions. In Section 2, we investigate part 1). We provide a full characterization of the pairs $(r,t)$ for which the answer is positive, with the sole exception of $(r,t)=(3,2)$, which remains open in any characteristic. Note that this is also one of the main unsolved cases when it comes to unimodality. We then propose and discuss in detail five open problems specifically on level Hilbert functions of codimension 3 and type 2, which we hope will be the object of future research in this area.

In Section 3, we address part 2). Our main results are a positive answer for all pairs $(r,1)$; negative for $(r,t)=(3,4)$; and negative for all pairs $(r,t)$ such that $r\ge 4$ and $2\le t\le r+1$. In all instances where log-concavity fails, it will be possible to exhibit some elegant, infinite families of pure $O$-sequences with the relevant parameters. Also interesting, the main case left open is again $(r,t)=(3,2)$. Further, we conjecture that, in analogy with the situation in the arbitrary level case, log-concavity fails for pure $O$-sequences of any given codimension $r\ge 3$ and type $t$ large enough.

We now briefly recall the main definitions. We refer to \cite{Ia2} and to standard texts such as \cite{St0} (for level Hilbert functions) and \cite{BMMNZ} (for pure $O$-sequences) for any basic facts or unexplained terminology.

We consider \emph{standard graded artinian algebras} $A=R/I$, where $R = k[x_1,\dots,x_r]$ is a polynomial ring over an infinite field $k$ (\emph{a priori,} we make no assumptions on the characteristic of $k$), $I\subset R$ is a homogeneous ideal, and all $x_i$ have degree one. Assuming without loss of generality that $I$ contains no nonzero linear forms, we define $r$ as the \emph{codimension} of $A$.

Recall that the \emph{Hilbert function} of $A$ is given by $h_i=\dim_k A_{i}$, for all $i\ge 0$. It is a standard fact that $A$ is artinian if and only if its Hilbert function is zero in all large degrees. Hence, we may identify the Hilbert function of $A$ with its \emph{$h$-vector} $h(A)=(1,h_1,\dots,h_e)$, where $e$ is the largest index such that $h_e>0$, and is called the \emph{socle degree} of $A$ (or of $h(A)$).

The \emph{socle} of $A$ is the annihilator of the maximal homogeneous ideal $(\overline{x_1}, \dots ,\overline{x_r})$. We define the \emph{socle-vector} of $A$ as $s=(0,s_1,\dots,s_e)$, where $s_i$ is the dimension over $k$ of the socle in degree $i$. We then say that $A$ is \emph{level} if $s=(0,\dots,0,t)$, i.e., if its socle is concentrated in the last degree. The integer $t=s_e=h_e$ is called the \emph{Cohen-Macaulay type} (or simply \emph{type}) of $A$.  Finally, $A$ is \emph{Gorenstein} if $A$ is level of type $t=1$.

We say that a Hilbert function $h=(1,h_1=r,h_2,\dots,h_e=t)$ is \emph{level} (of codimension $r$ and type $t$) if there exists a level algebra $A$ with $h(A)=h$.

A \emph{pure $O$-sequence} is the Hilbert function of a \emph{monomial} level algebra. Equivalently, in combinatorial terms, $h=(1,h_1=r,h_2,\dots,h_e=t)$ is a pure $O$-sequence (of codimension $r$ and type $t$) if there exist $t$ (monic) monomials of degree $e$ in $r$ variables having a total of $h_i$ monomial divisors in degree $i$, for all $i$.

For instance, $(1,3,4,2)$ is a pure $O$-sequence of codimension $3$, type $2$, and socle degree $3$, generated by the monomials $xyz$ and $xz^2$. Indeed, these two monomials have a total of four monomial divisors in degree 2 ($xy,xz,yz,z^2$), and three in degree 1 ($x,y,z$). 

Level Hilbert functions and pure $O$-sequences --- as well as much of the combinatorial commutative algebra that we know today --- find their inception in Richard Stanley's seminal work of the late Seventies \cite{St1,St2}. These two topics have since constituted an unusually active and successful research area of commutative algebra, to which this author also devoted at least a decade of his mathematical career. We refer to \cite{PSZ} and its references for a nonexhaustive list of contributions to the theory of level Hilbert functions since Richard's original manuscripts, and to the AMS Memoir \cite{BMMNZ} (and references thereof) specifically for pure $O$-sequences. We hope that this brief paper, including the open problems we present in the next two sections, can stimulate further work in this area. 

\section{Log-concavity of arbitrary level Hilbert functions}

The goal of this section is to characterize the pairs $(r,t)$ such that all level Hilbert functions of codimension $r$ and type $t$ are log-concave. We will achieve this with the sole exception of $(3,2)$, which remains open in any characteristic.  

We first need to review \emph{(Macaulay's) inverse systems}, also known as {\em Matlis duality}. For two good, in-depth introductions to this theory, we refer the reader to \cite{Ge,IK}. For simplicity's sake, in the brief description that follows we assume that the  characteristic of the base field $k$ is zero (or larger than the socle degree of the algebra), but the relevant parts of the theory, including Lemma \ref{lem1} and Proposition \ref{ia1} below, can be made characteristic-free using so-called \emph{divided powers} in lieu of differentials.  

Given a polynomial ring $R=k[x_1,\dots,x_r]$, consider the graded $R$-module $S=k[y_1,\dots,y_r]$, where $x_i$ acts on $S$ by partial differentiation with respect to $y_i$. This establishes a bijective correspondence between artinian algebras $R/I$ and finitely generated $R$-submodules $M$ of $S$, where $I=$ Ann$(M)\subset R$ is the annihilator of $M$, and $M=I^{-1}$ is the submodule of $S$ annihilated by $I$ (see \cite{Ge}, Remark 1, p. 17).

If $R/I$ has socle-vector $s=(0,s_1,\dots,s_e)$, then $M$ is (minimally) generated by $s_i$ homogeneous forms of degree $i$, for all $i=1,\dots,e$. Further, the Hilbert function of $R/I$ is given by the dimensions of the vector spaces spanned by the partial derivatives of the generators of $M$ in each degree (\cite{Ge}, Remark 2, p. 17). It follows that level algebras of type $t$ and socle degree $e$ correspond to $R$-submodules of $S$ generated by $t$ forms of degree $e$.

We only remark here that the equivalence between the algebraic and combinatorial definitions of pure $O$-sequences given earlier can be viewed as a special case of Macaulay's inverse systems: namely, when the modules $M$ are generated by monomials (since, clearly, the divisors of a given monomial coincide with its partial derivatives, up to constant factors).

We are now ready to state the following two technical results from \cite{Ia1}. They both hold characteristic-free. (We present the proposition below in the context of level algebras, which suffices for the purpose of this paper, but it remains true under more general hypotheses.)

\begin{lemma}[\cite{Ia1}, Proposition 4.7]\label{lem1}
Let $F=\sum_{i=1}^m L_i^e \in S=k[y_1,\dots,y_r]$, where the $L_i$ are \emph{general} linear forms (according to the Zariski topology). Then the Gorenstein algebra $R/Ann(F)$ dual to $F$ has Hilbert function
$$h(m)=(1,h_1(m),\dots,h_e(m)=1),$$
where, for $j=1,\dots,e$,
$$h_j(m)=\min \lbrace m,\dim_kR_j,\dim_kR_{e-j}\rbrace .$$
\end{lemma}
 
\begin{proposition}[\cite{Ia1}, Theorem 4.8 A]\label{ia1}
Let $h=(1,h_1,\dots,h_e)$ be the Hilbert function of a level algebra $A=R/I$, where $I$ annihilates the submodule $M$ of $S$. Let $m\leq \binom{r-1+e}{e}-h_e$, and consider a form $F\in S$ that is the sum of the $e$-th powers of $m$ general linear forms. (The Hilbert function, $h(m)$, of the Gorenstein algebra dual to $F$ is described in Lemma \ref{lem1}.) Then the level algebra dual to the module $\langle M,F \rangle $ has Hilbert function $H=(1,H_1,\dots,H_e),$ where, for $j=1,\dots,e$,
$$H_j=\min \left \lbrace h_j+h_j(m),\binom{r-1+j}{j} \right \rbrace .$$
\end{proposition}

We next record an elementary fact that will be used repeatedly in this paper.

\begin{lemma}\label{aaa}
Let $a$ be a positive integer, and consider three consecutive, maximal entries of a codimension 3 Hilbert function, say $h_{i-1}=\binom{i+1}{2}, h_i=\binom{i+2}{2}, h_{i+1}=\binom{i+3}{2}$. Then the sequence
$$h_{i-1}+a, {\ }h_{i}+a, {\ }h_{i+1}+a$$
is log-concave if and only if $a\le h_i$.
\end{lemma}

\begin{proof}
We want to determine the values of $a$ which satisfy the inequality:
$$\left(\binom{i+1}{2}+a\right) \left(\binom{i+3}{2}+a\right) \leq \left(\binom{i+2}{2}+a\right)^2.$$
A straightforward computation, that we leave to the reader, gives us that this is the case precisely when $a\leq \binom{i+2}{2}$.
\end{proof}

Finally, we recall Iarrobino's main result from \cite{Ia2}, which again holds characteristic-free.

\begin{theorem}\label{iarro}
\begin{enumerate}
\item All level Hilbert functions of codimension $2$, and all Gorenstein Hilbert functions of codimension $3$ are log-concave.
\item Log-concavity fails for Gorenstein Hilbert functions of codimension $4$. One example, which is not log-concave in degrees 9 through 13, is the following of socle degree 28. (See \cite{Ia2}, Example 3.5; by symmetry, here we only give the first half of $h$.)
$$h=(1,4,10,20,35,56,84,120,h_8=165,175,186,198,211,225,h_{14}=240,\dots,h_{28}=1).$$
\end{enumerate}
\end{theorem}

We are now ready for the main result of this section. We extend Theorem \ref{iarro} and classify \emph{nearly} all pairs $(r,t)$ that force log-concavity for all level Hilbert functions. The only exception, which we have not been able to determine and will be discussed in depth after the theorem, is $(r,t)=(3,2)$.

\begin{theorem}\label{main}
Let $r$ and $t$ be two positive integers. Then all level Hilbert functions of codimension $r$ and type $t$ are log-concave if and only if:
\begin{enumerate}
\item $r\le 2$; or
\item $r=3$ and $t\le t_0$, where $t_0=1$ or $2$.
\end{enumerate}
Moreover, with the possible exception of the case $(r,t)=(3,2)$, the above classification is independent of the characteristic.
\end{theorem}

\begin{proof}
The case $r=1$ is trivial, while $r=2$ was proven by Iarrobino in \cite{Ia2}, along with the case $(r,t)=(3,1)$ (see Theorem \ref{iarro}). Hence, not considering $(r,t)=(3,2)$, the first open case is $(3,3)$, for which we exhibit the following infinite family of non-log-concave level Hilbert functions. (Further infinite families can be constructed along the same lines.)

Consider the inverse system module $M\subset S=k[y_1,y_2,y_3]$ generated by three forms $F, G_1,G_2$ of large socle degree $e$ (any $e\ge 10$ will do), where  $F$ is {general} and each $G_i$ is the sum of the $e$-th powers of six general linear forms.

By Lemma \ref{lem1} and Proposition \ref{ia1}, if we consider the level algebra dual to $M$, we get that the last five entries (degrees $e-4$ through $e$) of its Hilbert function $h$ are given by:
$$(\dots, 15, 10, 6, 3, 1) +$$
$$2 \times (\dots, 6, 6, 6, 3, 1)=$$
$$h=(\dots, h_{e-4}=27, 22, 18, 9, h_e=3).$$

Thus, $h$ is not log-concave by Lemma \ref{aaa}, since $2\times 6=12 >10$, as desired. (Note that $27 \times 18 > 22^2$.) We only remark that, in this example, provided we assume $e\ge 12$, choosing $F$ to be the monomial $y_1^4y_2^4y_3^{e-8}$ would be general enough for our purposes.

A similar idea also allows us to disprove log-concavity for all types $t\ge 3$, when $r=3$. Indeed, consider the inverse system module $M= \langle F,G_1,\dots,G_{t-1}\rangle$, where $F$ is general and the $G_i$ are again the sum of the $e$-th powers of six general linear forms each. Then, for any socle degree $e$ large enough ($e\ge \left(5+\sqrt{48t+73}\right)/2$ will suffice), it is easy to see that the last five entries of the Hilbert functions $h$ of the level algebra dual to $M$ are:
$$h=(\dots, h_{e-4}= 15+6(t-1), 10+6(t-1),6+6(t-1),3+3(t-1),h_e=t).$$
Hence, by Lemma \ref{aaa}, $h$ always fails log-concavity in degree $e-3$, since
$$6(t-1)>10$$
for $t\ge 3$.

Recall that Weiss \cite{We}, using ideas developed by Iarrobino and this author \cite{Za1}, constructed examples of nonunimodal level Hilbert functions of codimension 3 and any type $t\ge 5$, and of codimension 4 and any $t\ge 3$. Thus, since log-concavity is a stronger property than unimodality, Weiss' results provide an alternative proof of the failing of log-concavity in codimension 3, when $t\ge 5$.

Further, unimodality is known to fail in codimension $r\ge 5$ and any type $t\ge 1$ (see Stanley \cite{St2} for the first nonunimodal Gorenstein example, and Bernstein-Iarrobino \cite{BI} for the first one specifically in codimension $5$; a brief selection of subsequent results can be found in \cite{AS,Bo,BL,MNZ1,MNZ2,MZ1}). Combining this with Weiss' work \cite{We}, we see that the only remaining case to prove in this theorem is the non-log-concavity of level Hilbert functions of codimension 4 and type 2.

A direct way to address this is to use Iarrobino's Gorenstein examples from \cite{Ia2}, combined with Lemma \ref{lem1} and Proposition \ref{ia1} (though several other constructions are possible; see for instance the next section for an infinite family of \emph{monomial} examples). Let $F\in S=k[y_1,\dots,y_4]$ be a form whose dual Gorenstein algebra of socle degree 28 has the Hilbert function displayed in Theorem \ref{iarro}, namely:
$$h=(1,4,10,20,35,56,84,120,h_8=165,175,186,198,211,225,h_{14}=240,\dots,h_{28}=1).$$
Now let $G=L^{28},$ where $L\in S$ is a general linear form, and consider the level algebra $A$ dual to the inverse system module $\langle F,G\rangle$. It is easy to see, by Lemma \ref{lem1} and Proposition \ref{ia1}, that the Hilbert function $H$ of $A$ can be obtained by adding 1 to $h$ in all degrees $i=9,\dots,28$. That is,
$$H=(1,4,10,20,35,56,84,120,165,H_9=176,H_{10}=187,199,212,H_{13}=226,241,$$$$H_{15}=226,212,199,187,H_{19}=176,166,121,85,57,36,21,11,5,H_{28}=2).$$

Standard computations now give us that $H$ fails to be log-concave (in all degrees $i=10,\dots,13$ and $i=15,\dots,19$). As we mentioned earlier, because of the nonunimodality of level Hilbert functions of codimension 4 and any type $t\ge 3$  \cite{We} --- or alternatively, by constructions similar to the above that rely on Lemma \ref{lem1} and Proposition \ref{ia1} --- this completely settles log-concavity in codimension 4, and concludes the proof of the theorem.
\end{proof}

We now briefly turn to the codimension 3, type 2 case, which was the only one left open in Theorem \ref{main}. In fact, the case $(3,2)$ is also among the \lq \lq smallest'' for which we do not have a characterization of all possible level Hilbert functions. (See \cite{Ia1,Ia2004} for an explicit description of the level Hilbert functions of codimension $r=2$, and \cite{BuEi,St2,Za2} for Stanley's characterization of Gorenstein Hilbert functions when $r=3$.)

For brevity's sake, let $S_{3,2}$ denote the set of level Hilbert functions of codimension 3 and type 2. In Question \ref{q1}, we present a selection of problems and conjectures on $S_{3,2}$, which, along with log-concavity, we consider among the most interesting and consequential.

\begin{question}\label{q1}
\begin{enumerate}
\item \emph{Are all functions in $S_{3,2}$ unimodal?}\\
Of the few classes of level Hilbert functions of given type and codimension for which unimodality has not yet been decided, this is among the most important, along with Gorenstein Hilbert functions of codimension 4. Note that the corresponding case for pure $O$-sequences was settled in the positive in \cite{BMMNZ}.

\item \emph{Are all functions in $S_{3,2}$ {flawless}?}\\
Recall that $h=(1,h_1,\dots,h_e)$ is \emph{flawless} if $h_i\le h_{e-i}$, for all $i\le e/2$. We believe the answer to this question to be positive, but despite several attempts, so far we have not been able to prove it.\\
Also interesting, we note that a \emph{negative} answer would imply the existence of Gorenstein algebras of codimension 3 which fail the so-called \emph{strong Lefschetz property} (see \cite{BMMNZ2}, Remark 2.3 and Corollary 2.6, part (2)). The Lefschetz properties (weak and strong) can be viewed as a generalization of the Hard Lefschetz Theorem of algebraic geometry. In recent years, they have received a great deal of attention from specialists in this area of commutative algebra, with some encouraging results. While both Lefschetz properties have been shown to fail for Gorenstein algebras of codimension 3 in \emph{positive characteristic} (see, e.g., \cite{Cook}), they are still open in \emph{characteristic zero}, and we do not believe there is consensus on a conjecture.\\
Finally, flawlessness is known to hold for the subset of pure $O$-sequences (in fact, it does hold for \emph{all} pure $O$-sequences \cite{Hibi}).

\item \emph{Is the first half of any function in $S_{3,2}$ differentiable?}\\
We say that the first half of $h=(1,h_1,\dots,h_e)$ is \emph{differentiable} if
$$(\Delta h)_{\le \lfloor e/2 \rfloor}=\left(1,h_1-1,h_2-h_1,\dots, h_{\lfloor e/2 \rfloor}-h_{\lfloor e/2 \rfloor -1}\right)$$
is an $O$-sequence, i.e., it satisfies Macaulay's theorem \cite{Ma,St0}. In particular, a Hilbert function $h$ with this property is nondecreasing throughout its first half, a fact that has also not been established for $S_{3,2}$.\\
Some of the considerations we made in part (2) for flawlessness find an analog for this other natural (possible) property of level Hilbert functions of codimension 3 and type 2. In particular, a negative answer to this question would imply the failing of the \emph{weak} Lefschetz property for some Gorenstein algebras of codimension 3 (see \cite{BMMNZ2}, Remark 2.3 and Corollary 2.6, part (1)). This property, as we mentioned, is still open in characteristic zero; see \cite{BMMNZ2,HMNW,MZ0} for partial progress. Further, differentiability throughout the first half is again known for all pure $O$-sequences \cite{Hibi}.\\
Similarly to part (2), it is reasonable to believe in a positive answer also to (3).

\item \emph{Do all functions in $S_{3,2}$ satisfy the Interval Conjecture?}\\
In the level case, the \emph{Interval Conjecture} posits that if for a positive integer $\alpha $, $(1,h_1,\dots,h_i,\dots,h_e)$ and $(1,h_1,\dots,h_i+\alpha,\dots,h_e)$ are two level Hilbert functions that only differ in one degree $i$, then $(1,h_1,\dots,h_i+\beta,\dots,h_e)$ is also level, for all $\beta =0,1,\dots,\alpha$.\\
The Interval Conjectures (both in the level and the Gorenstein case) were proposed by this author in 2009, with the goal of determining some strong and natural structural properties for the set of all level Hilbert functions \cite{Za3}. However, they have so far resisted all attempts of a proof (or a counterexample). For progress in the Gorenstein case, see our recent paper with Park and Stanley \cite{PSZ}.\\
Finally, over the years, with a few coauthors we explored the \lq \lq Interval Property'' for other sequences of interest in algebra and combinatorics, and in some instances we were able to obtain interesting results (see, e.g., \cite{BMMNZ,HSZ,PZ,SZ1}).

\item \emph{Are the functions in $S_{3,2}$ independent of the characteristic of the base field?}\\
We are not aware, to date, of any Hilbert function which is level (or Gorenstein) in one characteristic but not in another. In fact, we conjectured that this is never the case (\cite{MZ}, Conjecture 3.9). Note that, with partial exceptions (e.g., some of the techniques used to study the Lefschetz properties), the characteristic $p$ methods currently available in this area have rarely proven more powerful than the tools used in characteristic zero. Therefore, as a byproduct, in order to investigate this question it might desirable to develop new ideas for the study of level Hilbert functions in positive characteristic. 
\end{enumerate}
\end{question}

\section{Log-concavity of pure $O$-sequences}

We now turn to pure $O$-sequences. Here, unimodality is still mostly open, even though it has been settled in some instances. In particular, pure $O$-sequences of type 1 (which coincide with the Hilbert functions of \emph{complete intersection} algebras) are all unimodal, as one can see, e.g., via a simple generating function argument. Also, unimodality is known to hold for pure $O$-sequences of codimension $r=2$ and any type $t$ (since it does hold for the larger class of arbitrary level Hilbert functions with the same parameters), as well as in the specific cases $(r,t)= (3,2), (3,3), (4,2)$ \cite{BMMNZ,boy0,boy1,boy2}. To our knowledge, no other pair $(r,t)$ has so far been proven to force unimodality for all pure $O$-sequences. In particular, $(r,2)$ is still open, for any $r\ge 5$.

Similarly to the case of unimodality, deciding log-concavity for pure $O$-sequences is often harder than it is for arbitrary level Hilbert functions. However, we will be able to show that, while log-concavity always holds in type 1, it fails in very low type in all codimensions $r\ge 3$ (in particular, it fails in type 2 for any $r\ge 4$). Interestingly, the main case we leave open is again $(r,t)=(3,2)$, along with $(r,t)$ for $r\ge 3$ and $t$ large enough, where we conjecture that non-log-concave pure $O$-sequences always exist.

Finally, recall that even though pure $O$-sequences are defined algebraically as Hilbert functions of monomial level algebras, and the properties of these rings may heavily depend on the characteristic of the base field, pure $O$-sequences can also be described entirely combinatorially, in terms of the number of divisors of certain monomials. Therefore, many numerical properties of these sequences (including log-concavity, unimodality, etc.) are guaranteed to be \emph{independent} of the characteristic.

We begin by recording the simple case of type 1.

\begin{proposition}\label{type1}
All pure $O$-sequences of type 1 are log-concave, in any codimension $r$.
\end{proposition}

\begin{proof}
Let the type $1$ pure $O$-sequence $(h_0=1,h_1,\dots,h_e=1)$ be generated by $x_1^{a_1}x_2^{a_2}\cdots x_r^{a_r}$. It is easy to see that the $h_i$ are determined by the following generating function identity:
$$\sum_{i=0}^e h_iq^i = \prod_{i=1}^r \left(1+q+\dots +q^{a_i}\right).$$
Thus, since the product of log-concave polynomials (with positive coefficients) is log-concave, the result immediately follows.
\end{proof}

We now present the main theorem of this section, which takes care of pure $O$-sequences in low type for all codimensions $r$, with the partial exception of $r=3$.

\begin{theorem}\label{theo}
There exist infinite families of non-log-concave pure $O$-sequences of codimension $r$ and type $t$ when:
\begin{enumerate}
\item $r=3$ and $t=4$;
\item $r\ge 4$ and $t=2,3,\dots,r+1$.
\end{enumerate}
\end{theorem}

\begin{proof}
We first show the case $(r,t)=(3,4)$, by constructing a non-log-concave pure $O$-sequence for any socle degree $e$ large enough ($e\ge 12$ will suffice). Consider the following four monomials in $x,y,z$:
$$x^4y^4z^{e-8},  {\ }x^{e-2}yz,  {\ }xy^{e-2}z,  {\ }xyz^{e-2}.$$

Standard computations show that the last five entries (degrees $e-4$ through $e$) of the pure $O$-sequence $h$ they generate are given by:
$$(\dots, 15, 10, 6, 3, 1) +$$
$$3 \times (\dots, 4, 4, 4, 3, 1)=$$
$$h=(\dots, h_{e-4}=27, 22, 18, 12, h_e=4).$$
Thus, $h$ is not log-concave by Lemma \ref{aaa}, since $3\times 4=12 >10$. (Of course, various other constructions are possible using the same idea.)

Now let $r= 4$. We first exhibit an infinite, non-log-concave family of type $t=2$. Let $F$ be any monomial in $x,y,z$ of degree $e$ large enough (any $e\ge 42$ will do), which behaves like a \emph{general} form in degrees 1 through 14. For instance,
$$F=x^{14}y^{14}z^{e-28}.$$
Next, we introduce a new variable $t$, and let
$$G=x^{e-12}y^4z^4t^4.$$
Note that $F$ and $G$ have no common divisors in any degree $i\ge e-14$.

By symmetry of type 1 pure $O$-sequences, it is easy to see that the last 15 entries (degrees $e-14$ through $e$) of the pure $O$-sequence $h$ generated by $F$ and $G$ are given by:
$$(\dots, 120,105,91,78,66,55,45,36,28,21,15,10,6,3,1) +$$
$$(\dots, 125,125,125,124,121,115,105,90,72,53,35,20,10,4,1)=$$
$$\left(\dots, h_{e-14}=120+125=245,h_{e-13}=105+125=230,h_{e-12}=91+125=216,\dots,h_e=2\right).$$
Hence, $h$ fails to be log-concave in degree $e-13$ by Lemma \ref{aaa}, since $125>105$.% (note that $245 \times 216> 230^2$).

We now show that the above non-log-concave construction for type 2 easily extends to all types $t\le 5$. Indeed, each of the monomials
$$G_2=x^4y^{e-12}z^4t^4, {\ }G_3=x^4y^4z^{e-12}t^4, {\ }G_4=x^4y^4z^4t^{e-12}$$
generates the same pure $O$-sequence as $G=G_1$. Also, since $e$ is large enough, the $G_j$ have no common divisors among each other or with $F$, in any degree $i\ge e-14$.

Therefore, if we add one $G_j$ at the time ($j\ge 2$) to the type 2 generating set given by $F$ and $G$, the new pure $O$-sequences we obtain (of type 3, 4, and 5, respectively) again fail log-concavity in degree $e-13$, by Lemma \ref{aaa}.

Now let $r\ge 5$. We generalize the above construction to produce infinite families of non-log-concave pure $O$-sequences of type 2 and codimension $r$. Consider a socle degree $e$ large enough (any $e\ge 3r+30$ will do), and pick a monomial $F$ in $x_1,x_2,x_3$ which behaves like a general form in degrees 1 through $r+10$. For instance,
$$F=x_1^{r+10}x_2^{r+10}x_3^{e-2r-20}.$$
Next, we introduce new variables $x_4,\dots,x_r$, and let
$$G=x_1^{e-r-8}x_2^4x_3^4x_4^4x_5\cdots x_r.$$

It is a simple exercise that we leave to the reader to show that if a monomial $T$ in variables $x_1,\dots,x_{r-1}$ generates a pure $O$-sequence $h=(1,h_1,\dots,h_s=1)$, then the monomial $T\cdot x_r$ generates $H=(1,H_1,\dots,H_s,H_{s+1}=1)$, where
$$H_i=h_i+h_{i-1}$$
for all $i\le s$.

By the unimodality of type 1 pure $O$-sequences and the assumption that $e$ is large enough, we deduce from the proof of the  case $r=4$ that the pure $O$-sequence generated by
$$x_1^{e-r-8}x_2^4x_3^4x_4^4$$
equals 125 in all degrees $\lfloor e/2 \rfloor$ through $e-12$. Thus, by the above exercise, the pure $O$-sequence generated by
$$x_1^{e-r-8}x_2^4x_3^4x_4^4x_5$$
equals $125+125=250$ in all degrees leading up to, and including, $e-13$. By iterating this process, one moment's thought gives us that the pure $O$-sequence $H$ generated by
$$G=x_1^{e-r-8}x_2^4x_3^4x_4^4x_5\cdots x_r$$
satisfies
$$H=(\dots, H_{e-r-10}=125\times 2^{r-4}, H_{e-r-9}=125\times 2^{r-4}, H_{e-r-8}=125\times 2^{r-4}, \dots).$$

Next observe that, since $e\ge 3r+30$, the codimension 3 pure $O$-sequence $h'$ generated by $F$ is maximal in degrees $i\le r+10$; that is,
$$h'_i=\binom{i+2}{2},$$
for all $i\le r+10$.

Thus, by symmetry, we have:
$$h'=\left(\dots, h'_{e-r-10}=\binom{r+12}{2}, h'_{e-r-9}=\binom{r+11}{2}, h'_{e-r-8}=\binom{r+10}{2}, \dots \right).$$

Finally, notice that, by construction, the two monomials $F$ and $G$ have no common divisors in any degree $i\ge e-r-10$.

Therefore, the relevant entries of the pure $O$-sequence $h$ generated by $F$ and $G$ are given by:
$$h_{e-r-10}=\binom{r+12}{2}+125\times 2^{r-4}, {\ }h_{e-r-9}=\binom{r+11}{2}+125\times 2^{r-4},$$$$ h_{e-r-8}=\binom{r+10}{2}+125\times 2^{r-4}.$$
Hence, by Lemma \ref{aaa}, $h$ fails log-concavity in degree $e-r-9$, since
$$125\times 2^{r-4} > \binom{r+11}{2}$$
whenever $r\ge 4$. This concludes the proof of non-log-concavity for type 2.

We now sketch a proof of the fact that, similarly to the case $r=4$, the above construction extends to produce infinite families of non-log-concave pure $O$-sequences of type $t=3,\dots, r+1$, in any codimension $r\ge 5$.

For $j=2,\dots,r$, we define $r-1$ new monomials $G_j$ by permuting $x_1$ with $x_j$ in $G=G_1$, and leaving the other variables unchanged. Hence, in any $G_j$, the variable $x_j$ appears with exponent ${e-r-8}$, and all other variables with exponent 1 or 4. It is clear that each $G_j$ generates the same pure $O$-sequence. Further, the $G_j$ have no common divisors among each other or with $F$, in any degree $i\ge e-r-10$. 

Therefore, by adding one $G_j$ at the time ($j\ge 2$) to the type 2 generating set given by $F$ and $G$, the same reasoning as above combined with Lemma \ref{aaa} again guarantees the failing of log-concavity in degree $e-r-9$, for any $t\le r+1$. This completes the proof of the theorem.
\end{proof}

We wrap up this paper with the following multi-part question, where we suggest a few more problems for future research.

\begin{question}\label{q2}
\begin{enumerate}
\item \emph{Are all pure $O$-sequences of codimension 3 and type 2 log-concave?}\\
This is perhaps the main case left open by Theorem \ref{theo}. Interestingly, as we saw in the previous section, the corresponding case for arbitrary level Hilbert functions is the only one that remains unsolved in that context.\\
Also note that both flawlessness and differentiability throughout the first half are known to hold for \emph{any} pure $O$-sequence \cite{hausel,Hibi}. Further, Iarrobino's argument for the log-concavity of Gorenstein Hilbert functions of codimension $r=3$ \cite{Ia2} essentially boiled down to showing that differentiable sequences are always log-concave for $r=3$. From this, it easily follows that all pure $O$-sequences of codimension 3 are log-concave throughout the first half. Therefore, what remains to be understood --- which is often the case when one studies pure $O$-sequences --- is their behavior in the second half.\\
Finally, a positive answer to this question would provide a new, natural proof of the unimodality of pure $O$-sequences of codimension 3 and type 2. This result was first shown in \cite{BMMNZ} as a consequence, in \emph{characteristic zero}, of the weak Lefschetz property for the corresponding monomial level algebras.

\item \emph{For any given $r\ge 3$, does log-concavity for pure $O$-sequences fail in type $t$ for all $t$ large enough?}\\
We believe the answer is positive. It does not seem difficult the generalize the non-log-concave constructions of Theorem \ref{theo} to other small types $t$, for any $r\ge 3$. However, the lack of general tools available to extend basic properties of pure $O$-sequences from one type to the next --- including the failing of log-concavity or unimodality --- appears to make this question nontrivial. Note that this is in stark contrast to the case of arbitrary level Hilbert functions, where results such as Lemma \ref{lem1} and Proposition \ref{ia1} could extensively be employed.

\item \emph{Assuming a positive answer to (2), what is the smallest $t_0=t_0(r)$ such that log-concavity fails for pure $O$-sequences of codimension $r\ge 3$ and \emph{any} type $t\ge t_0$?}\\
One might even wonder whether, given $r$, no \lq \lq gaps'' occur for the types $t$ where log-concavity fails. If this is the case, note that $2\le t_0(3)\le 4$, and $t_0(r)=2$ for all $r\ge 4$.
\end{enumerate}
\end{question}

\section*{Acknowledgements} We wish to thank the anonymous referee for a careful reading of our manuscript and several helpful comments. We also warmly thank Tony Iarrobino for alerting us to his preprint \cite{Ia2} and for subsequent email exchanges on this topic. We are grateful to our former student Sung Gi Park (Harvard) for several discussions on pure $O$-sequences (in person back at MIT in 2017 and via email most recently). This work was partially supported by a Simons Foundation grant (\#630401).

\end{document}